\documentclass[12pt]{article}

\usepackage{amssymb,latexsym}
\usepackage{amsmath,amscd}
\usepackage{theorem}
\usepackage[all]{xy}
\usepackage{enumerate}
\usepackage[english]{babel}

\usepackage{color}

\newtheorem{lemma} {Lemma} [section]

\newtheorem{theorem} [lemma] {Theorem}
\newtheorem{cor} [lemma] {Corollary}

\theorembodyfont{\rm} 

\newtheorem{remark}[lemma]{Remark}

\newenvironment{proof}{{\sc Proof:}}{{\hspace*{\fill} $\square$\\}}

\numberwithin{equation}{section}

\title{\bf Weierstrass points on Kummer extensions}
\date{}
\author{Miriam Abd\'on\thanks{IME, Univ. Federal Fluminense, Rua M\'ario Santos Braga s/n, CEP 24.020-140, Niter\'oi,Brazil
(miriam@mat.uff.br).}\and 
Herivelto Borges\thanks{Inst. de Ci\^encias Matem\'aticas e de Computa\c c\~ao, Universidade de S\~ao Paulo, Av. Trabalhador S\~ao Carlense, CEP 13560-970, S\~ao Carlos, Brazil(hborges@icmc.usp.br).}\and
Luciane Quoos\thanks{Instituto de Matem\'atica, Universidade Federal do Rio de Janeiro, Cidade Universit\'aria, CEP 21941-909,  Rio de Janeiro, Brazil (luciane@im.ufrj.br).\newline The second author was partially supported  by FAPESP-Brazil grant 2011/19446-3.}}

\def\F{\mathbb F}
\def\P{\mathbb P}
\def\N{\mathbb N}
\def\Z{\mathbb Z}

\def\cA1{{\mathcal A_1}(P_\infty)}

\def\cD{\mathcal D}

\def\cL{\mathcal L}

\def\P{\mathbb P}
\def\Fq{{\mathbb{F}_q}}
\def\Fq2{{\mathbb F}_{q^2}}

\begin{document}
\maketitle

\begin{abstract}
\maketitle

\noindent
For Kummer extensions  $y^m=f(x)$, we discuss conditions for an integer be a Weierstrass gap at a place $P$. In the case of totally ramified places, the conditions will be necessary and sufficient. As a consequence, we extend independent results of several authors.

\vspace{0.5cm}
\noindent
MSC2010 Subject Classification Numbers: Primary 14H55, 11R58; Secondary 14Hxx.
\end{abstract}

\section{Introduction}\label{sec:intro}

Let $F$ be an algebraic function field in one variable defined over an algebraically closed field $K$ of characteristic $p\geq 0$. For any given place $P$ of $F$, a number $s$ is called a non-gap at $P$ if there exists
a function $z \in F$ such that the pole divisor of $z$ is $(z)_{\infty}=sP$. If there is no such function, the number $s$ is called a gap at $P$. If $g$ is the genus of $F$, it follows from the Riemann-Roch theorem that there are exactly $g$ gaps at any  place $P$ of $F$ \cite{Sti}.
It is a well known fact that for all, but finitely many places $P$, the gap sequence (of $g$ numbers) is always the same. This sequence is called the {\em gap sequence} of $F$. When the  gap sequence of $F$ is $(1, 2,\dots,g)$, the sequence is called {\em ordinary}, and we say that  the function field $F$ is  {\em classical}, otherwise $F$ is called nonclassical. The places for which the gap sequence is not equal to the gap sequence of $F$ are called {\em Weierstrass}.

Since Weierstrass proved his  L\"uckensatz (or gap theorem) in the early 1860s, the theory of Weierstrass points has contributed to the development
of several areas. The study of automorphisms of algebraic curves (e.g. \cite{Hurwitz},\cite{Lew90}), as well as  the number of rational points on curves over finite fields (e.g. \cite{StVo}, \cite{GM})  are classical examples. The theory of algebraic geometric codes is another  (recent) example of  area benefitting  from the  study of  Weierstrass points (e.g. \cite{gretchen}, \cite{carvalho}). 
For a more complete and detailed account on the applications of this theory, we refer to the survey paper  \cite{andrea}. 

%

The characterization of  Weierstrass points is a fundamental problem in the theory of algebraic curves. There are some situations in which the fully ramified places of $F/K(x)$ are known to be Weierstrass. For example, if the Kummer extension $y^m=f(x)$ is classical,  and $f(x)$ has at least five roots then, by Lewitte's results in \cite{Lew},  all branch points are Weierstrass. Another example would be the case of cyclic extensions $F/K(x)$ of degree $p^n$, where $char(K)=p$ and $n \geq 2$. In \cite{V-M}, Valentini and Madan proved that for such extensions  all fully ramified places are Weierstrass points.  The result remains valid for certain elementary abelian extensions of $K(x)$ \cite{garcia}. In general, when $F$ is classical, the number of Weierstrass points counted with their weight is $g^3-g.$ Estimating the exact weight of Weierstrass points is another  challenging problem in the  theory.

In this paper, we will be interested in providing conditions for an integer $s$ be a gap at a place $P$ in a Kummer extension $F/K(x).$  These conditions are derived from  our main result, Theorem \ref{main}. As a consequence,  independent results of several authors will be extendend.

Our first  application of Theorem \ref{main} is in the theory of maximal curves over finite fields.
In 2008, using  Cartier Operators, Garcia and Tafazolian  \cite{GT2} proved that if  the Fermat curve $y^m=1-x^m$ is  maximal over $\mathbb{F}_{q^2}$,  then $m$  must be a  divisor of $q+1$. The converse is a well known result.  In fact, since the Fermat curve is covered by the Hermitian curve $x^{q+1}+y^{q+1}+1=0$ over $\mathbb{F}_{q^2}$, the $\mathbb{F}_{q^2}$-maximality  follows from a result J. P. Serre  \cite{La}. In 2010, Tafazolian used results of supersingular curves and simplified the proof given in 2008 \cite{T3}. A natural question that arises is whether this result
can be extended to a larger family of curves of type $y^m=f(x)$. This question was somewhat  addressed in subsequent papers by  Garcia, Tafazolian and Torres (\cite{GT1},\cite{T3},\cite{T4},\cite{TT}), where the same result was proved when $1-x^m$ is replaced by some other polynomials $f(x) \in \mathbb{F}_{q^2}[x]$. In section \ref{maximal}, we present a  general class of polynomials $f(x) \in \mathbb{F}_{q^2}[x]$ for which the $\mathbb{F}_{q^2}$-maximality of  $y^m=f(x)$ implies $m|(q+1)$. This will subsume most of the previous results.

In the 1950's,  Hasse \cite{Ha} computed the Weierstrass weight at  totally ramified points of the classical Fermat curves $y^m=x^m-1, m\geq 2$. We apply our results to show how this can be easily extended to classical Kummer curves  $y^m=f(x)$, where $f(x)\in K[x]$ is separable. For this same type of classical curve, since the number of Weierstrass points counted with multiplicity is $g^3-g$, Towse   \cite[Corollary 10]{To} computed a lower bound for
$$\lim_{d \to \infty} \frac{BW}{g^3-g},$$ 
\noindent where  $BW$ is the Weierstrass weight of all totally ramified places, $g$ is the genus of the curve and $d$ is the degree of $f(x)$. In this paper we find the exact value of this limit.

For the same curves $y^m=f(x)$, where $f(x)$ is a separable polynomial, we show that $1, 2, \dots, m-2$ are gaps at  the generic points. This  extends a result of Leopoldt \cite{Leo}.

In \cite[Theorem 3]{V-M}, Valentini and Madan proved that in a Kummer extension $y^m=f(x)$ of degree $m$, if at least $m+3$ places of $F$ are fully ramified, then every totally ramified place is Weierstrass. We  extend the result by replacing the number of fully ramified places with the number of roots of $f(x).$ 

We also point out that the classicality of  hyperelliptic curves is established as an immediate  consequence of  Theorem \ref{main}.  That is actually a retrieval of a result due to F. K. Schmidt \cite{Sch}.

 This paper is organized as follows. In Section $2$, we set up the notation and present  some  preliminary results. In Section $3$, we will build on a result of Maharaj \cite{M} to prove Theorem \ref{main}. In Section $4$, as a first application, we establish  general conditions for which the   $\F_{q^2}$-maximality of  the curve $y^m=f(x)$ implies $m|(q+1)$. This improves upon  recent results in \cite{GT1},\cite{T3},\cite{T4} and \cite{TT}. In Section $5$, we present a list of sufficient conditions for  which all totally ramified places are Weierstrass. We end our discussion in Section \ref{sep}, where Theorem \ref{main} is used in the particular case where $f(x)$ is a separable polynomial.  This case is the source of many of our extended results.

\section{Notations and Preliminaries}

We begin by establishing some notations (see \cite{Sti} for more details).

Let $F/K$ be an algebraic function field in one variable, where $K$ is algebraically closed of characteristic $p\geq 0$. We denote by $\P(F)$ the set of places of $F$ and by $\cD_F$ the (additively written) free abelian group generated by the places of $F$. The elements $D$ in $\cD_F$ are called {\it divisors} and can be written as,
$$
D=\sum_{P\in \P(F)}n_P\,P\quad \text{ with } n_P\in \Z, \text{ almost all } n_P=0.
$$
The degree of a divisor $D$ is $deg(D)=\sum\limits_{P\in \P(F)}n_P.$
Given a divisor $D$ in $\cD_F$, the Riemann-Roch vector space associated to $D$ is defined by
$$
\cL(D):=\{z\in F\,|\, (z)\ge -D\}\cup \{0\}.
$$
We will denote by $\ell(D)=dim\, \cL(D)$ as vector space over the field of constants $K$. From  Riemann-Roch it follows that 
$$
\ell(D) = deg(D)+1-g \text{ if } deg(D)\ge 2g-1, 
$$
where $g$ is the genus of $F$.
Let $P$ in $\P(F)$, the Weierstrass non-gap semigroup at $P$ is:
$$H(P)=\{s\in \N_0\,|\,  (z)_\infty=sP \text{ for some } z\in F\} \cup \{0\}.$$ 
A non-negative integer $s$ is called a non-gap at $P$ if $s\in H(P)$, and a gap otherwise.
The Weierstrass Gap Theorem asseverates that $s$ is a gap at a place of degree one $P$ if and only if $\ell((s-1)P)=\ell(sP)$. As a consequence there exists $g$ gaps, $1=\lambda_1<\cdots<\lambda_g\le 2g-1$.
For any field $K \subseteq E\subseteq F$, writing the divisor $D$ of $F$ as 
$$
D = \sum_{R\in \P(E)}\,\sum_{\substack{{Q\in \P(F)}\\{Q|R}}}\, n_Q\, Q.
$$
We define the restriction of $D$ to $E$ as the divisor $D|_E$ in $E$ given by:
$$
D_{|E}:= \sum_{R\in \P(E)} min\,\bigg\{\bigg\lfloor\frac{n_Q}{e(Q|R)}\bigg\rfloor: 
{Q|R}\bigg\}\,R,
$$
where $e(Q|R)$ is the ramification index of $Q$ over $R$.

Hereafter a {\it Kummer extension} $F/K(x)$ is a field extension defined by 
$$y^m=f(x)=\prod_{i=1}^{r}(x-\alpha_i)^{\lambda_i}, \text{ with } m \geq 2,\, p\nmid m,\, \text{ and } 0 <\lambda_i<m,$$
where $f(x) \in K[x]$ is not a $d^{th}$ power (where $d\mid m$) of an  element in $K(x)$, and  $\alpha_1,\ldots,\alpha_r$ are pairwise distinct.

Throughout the paper, $\gcd(a, b)$ will be denoted by $(a,b)$.

The next result  (see \cite{M}, Theorem 2.2.) will be the main ingredient in the proof of Theorem \ref{main}.

\begin{theorem}[Maharaj] \label{ThMaharaj}
Let $F/K(x)$ be a Kummer extension of degree $m$ defined by $y^m=f(x)$.  Then for any divisor $D$ of $F$ which is invariant by the action of $Gal(F/K(x))$, we have that
$$ \mathcal{L}(D)= \bigoplus\limits_{t=0}^{m-1} L([D+(y^t)]_{|K(x)})\,y^t,$$
where $[D+(y^t)]_{|K(x)}$ denotes the restriction to $K(x)$ of the divisor $D+(y^t).$
\end{theorem}

\section{Main result}

In this Section we prove our main result, Theorem \ref{main}. We begin collecting some simple facts which will be used throughout our computations. 
The floor, ceiling and fractional part functions, will be denoted by $\Big\lfloor \cdot \Big\rfloor,\Big\lceil \cdot \Big\rceil$ and $\Big\lbrace \cdot \Big\rbrace$ respectively.

\begin{remark}\label{lulu} Given $x,x_1,\ldots,x_k \in \mathbb{R}$, and integers $ m \geq 1, a, \lambda$. The following holds: 
\begin{enumerate}
\item[(i)]$ \Big\lceil-x\Big\rceil=-\Big\lfloor x\Big\rfloor.$
\item[(ii)]$\Big\lceil \sum\limits_{i=1}^{k} x_i\Big\rceil=\sum\limits_{i=1}^{k}\Big\lfloor  x_i \Big\rfloor+\Big\lceil\sum\limits_{i=1}^{k}\Big\lbrace x_i\Big\rbrace\Big\rceil.$
\item[(iii)] If $a\equiv\, r \mod\, m$, with $0\leq r \leq m-1$, then $\Big\lbrace\frac{a}{m}\Big\rbrace=\frac{r}{m}.$
\item[(iv)] If $(\lambda, m)=1$, then for any integer $n$ there exists a unique  $t\in \{0, \dots, m-1\}$ such that $t\lambda \equiv\, n \, \mod \, m$.
\end{enumerate}
\end{remark}

\begin{lemma}\label{inteiro}
If  $\lambda_1,\ldots,\lambda_r \in \{1,\ldots,m-1\}$, then 
\begin{enumerate}
\item[(i)] there exists $t\in \{1,\ldots,m-1\}$ such that
$$ \sum\limits_{i=1}^r\Big\lbrace\frac{t\lambda_i}{m}\Big\rbrace \geq r/2.$$
\item[(ii)] If  $\lambda_1=\cdots=\lambda_r$, and $(\lambda_1, m)=1$, then there exists $t\in \{1,\ldots,m-1\}$ such that
$$ \sum\limits_{i=1}^r\Big\lbrace\frac{t\lambda_i}{m}\Big\rbrace =r(m-1)/m.$$
\end{enumerate}

\end{lemma}
\begin{proof}
Note that for any $\alpha \in \{1,\ldots,m-1\}$, with $(\alpha,m)=1$, we have $\Big\lbrace\frac{-\alpha\lambda_i}{m}\Big\rbrace=1-\Big\lbrace\frac{\alpha\lambda_i}{m}\Big\rbrace$, and then
$$\sum\limits_{i=1}^r\Big\lbrace\frac{(m-\alpha)\lambda_i}{m}\Big\rbrace=\sum\limits_{i=1}^r\Big\lbrace-\frac{\alpha\lambda_i}{m}\Big\rbrace=r-\sum\limits_{i=1}^r\Big\lbrace\frac{\alpha\lambda_i}{m}\Big\rbrace,$$
that is, 
$$\sum\limits_{i=1}^r\Big\lbrace\frac{(m-\alpha)\lambda_i}{m}\Big\rbrace+\sum\limits_{i=1}^r\Big\lbrace\frac{\alpha\lambda_i}{m}\Big\rbrace=r.$$

Therefore, either $\sum\limits_{i=1}^r\Big\lbrace\frac{\alpha\lambda_i}{m}\Big\rbrace\geq r/2$ or $\sum\limits_{i=1}^r\Big\lbrace\frac{(m-\alpha)\lambda_i}{m}\Big\rbrace\geq r/2,$ and the result follows.

The item (ii) follows from Remark \ref{lulu}, choosing $n=m-1$.
\end{proof}

\begin{lemma}\label{lema-t}

If $s$ and $m$  are positive integers, and  $\lambda \in \mathbb{Z}$ is such that $(\lambda,m)=1$, then there exists a unique $t\in \{0,\ldots, m-1\}$  satisfying  
$$s+\lambda t =(\Big\lfloor\frac{\lambda t}{m}\Big\rfloor+\Big\lfloor\frac{s}{m}\Big\rfloor)m + m-1.$$
 Moreover, for any $i\in \mathbb{Z}$,
\begin{equation*}
\Big\lfloor\frac{s+1+i\,\lambda}{m}\Big\rfloor =
\begin{cases}
\Big\lfloor\frac{i\,\lambda}{m}\Big\rfloor+\Big\lfloor\frac{s}{m}\Big\rfloor+1=\Big\lfloor\frac{s+i\,\lambda}{m}\Big\rfloor+1, &\text{ if }
     i \equiv t \mod m\\
    \text{}\\ 
\Big\lfloor\frac{s+i\,\lambda}{m}\Big\rfloor,  &\text{ otherwise}.\, 
\end{cases}
\end{equation*}

\end{lemma}
\begin{proof}
Since  $(\lambda,m)=1$, by Remark \ref{lulu} there exists a unique $t$ in $\{0,\ldots,m-1\}$ such that $s+t\lambda\equiv\, m-1 \mod \, m.$ Set  $q=\Big\lfloor\frac{s+\lambda t}{m}\Big\rfloor$, and write $s=m\Big\lfloor\frac{s}{m}\Big\rfloor + s_0$ with $s_0\in \{0,\ldots,m-1\}$. Thus  $t\lambda=m(q-\Big\lfloor\frac{s}{m}\Big\rfloor)+m-1-s_0$, and since $m-1-s_0\in \{0,\ldots,m-1\}$, it follows that  $\Big\lfloor\frac{\lambda t}{m}\Big\rfloor=q-\Big\lfloor\frac{s}{m}\Big\rfloor=\Big\lfloor\frac{s+\lambda t}{m}\Big\rfloor-\Big\lfloor\frac{s}{m}\Big\rfloor,$ which proves the first assertion.
For the second  assertion, note that we may assume $i\in \{0,1,\ldots,m-1\}$. Since $t$ is unique, for any  value $i\in\{0,\ldots,m-1\}\backslash \{t\}$, we find that $s+\lambda i \equiv r \mod m$, for some $r\in\{0,\ldots, m-2\}$ which gives $\Big\lfloor\frac{s+1+\lambda i}{m}\Big\rfloor=\Big\lfloor\frac{s+\lambda i}{m}\Big\rfloor$, 
and the result follows.
\end{proof}

The following Lemma will be useful to establish some consequences of Theorem \ref{main}.

\begin{lemma}\label{dimension}

Let $F/K(x)$ be a Kummer extension given by 
$y^m=f(x)$. Let $P$ be a place of $K(x)$, and  $Q_1,\ldots, Q_r$ be all the distinct places of $F$ lying over $P$. For $s\geq 1$,  if $\ell(sQ_1)+1=\ell((s+1)Q_1)$, then 
$\ell (s\sum\limits_{j=1}^rQ_{j})+r=\ell((s+1)\sum\limits_{j=1}^rQ_{j}).$

\end{lemma}

\begin{proof}

Let $z_1\in F$ be a function such that $(z_1)_{\infty}=(s+1)Q_1$, and define $\mathcal{L}_0:=\mathcal{L}(s\sum\limits_{j=1}^rQ_{j})$ and 
 $\mathcal{L}_i:=\mathcal{L}((s+1)\sum\limits_{j=1}^iQ_{j}+s\sum\limits_{j=i+1}^rQ_{j})$, $i=1,\ldots,r$.  It is well known that $\ell_{i}-\ell_{i-1}\leq 1$, and we claim that
 $\ell_{i}-\ell_{i-1}=1$; that is, $\mathcal{L}_{i-1} \subsetneq \mathcal{L}_{i}$. In fact, since $G=Gal(F/K(x))$ acts transitively on the set $\{Q_1,\ldots,Q_r\}$, we can choose $\sigma_1,\ldots,\sigma_r \in G$  such that  $Q_i=\sigma_i(Q_1)$, and define $z_i=\sigma_i(z_1)$, for  $i=1,\ldots,r$.
  Thus $z_i \in \mathcal{L}_i \backslash \mathcal{L}_{i-1}$, which proves the claim. Therefore, since $\ell_{i}=\ell_{i-1}+1$, for $i=1,\ldots,r$,  we arrive at  $\ell_m-\ell_0=r$, as desired.

\end{proof}

\begin{theorem}\label{main}
Let $F/K(x)$ be a  Kummer extension given by 
$$y^m=\prod_{i=1}^{r}(x-\alpha_i)^{\lambda_i}, \text{ with } 0 <\lambda_i<m.$$
Fix an element $\alpha_0\in K\backslash\{\alpha_1,\ldots,\alpha_r\}$, and set $\alpha_{r+1}:=\infty$, $\lambda_0:=0$ and 
$\lambda_{r+1}:=-\sum\limits_{i=0}^k\lambda_i$.  For each $u=0,\ldots, r+1$, consider the usual place $P_u\in \mathbb{P}(K(x))$ corresponding to $\alpha_u\in K\cup \{\infty\}$, and  the  divisor $D_u=\sum\limits_{{Q_{uv}|P_u}}Q_{uv} \in \mathcal{D}_F$. 
 If $s \geq 1$ is an integer, then $\ell (sD_u)-\ell ((s-1)D_u)$
is the number of elements $j\in \{0,\ldots,(m,\lambda_u)-1\}$ such that 
\begin{equation}\label{eqmain}
\sum\limits_{i=0}^r\Big\lbrace\frac{(t_u+jm_u)\lambda_i}{m}\Big\rbrace\leq 1+\Big\lfloor\frac{s-1}{m_u}\Big\rfloor,
\end{equation}
where $m_u=\frac{m}{(m,\lambda_u)}$, and  $t_u \in \{0,\ldots,m_u-1\}$ is the unique  element  for which  $s+\frac{\lambda_u}{(m,\lambda_u)}t_u \equiv 0\mod m_u$.

\end{theorem}
\begin{proof} Since $Gal(F/K(x))$ fixes  the divisor $D_u$,  it follows from Theorem \ref{ThMaharaj} that 
$$\mathcal{L}(sD_u)= \bigoplus_{t=0}^{m-1} \mathcal{L}([sD_u+(y^t)]_{|K(x)})\,y^t.$$
One can easily check that  $(y)= \sum\limits_{i=0}^{r+1} \sum\limits_{\\{Q_{iv}|P_i}}\frac{\lambda_i}{(m, \lambda_i)}\,Q_{iv}$ is divisor of the function $y$.  Therefore
$$sD_u+(y^t)= \Big(\frac{(m,\lambda_u)s+t\lambda_u}{(m,\lambda_u)}\Big)\,D_u +\sum_{\substack{{i=0}\\{i\neq u}}}^{r+1} \sum\limits_{\\{Q_{iv}|P_i}}\frac{t\lambda_i}{(m, \lambda_i)}\,Q_{iv},$$
and restricting this divisor to $K(x)$, we obtain 
\begin{eqnarray*}
[sD_u+(y^t)]_{|K(x)}&=& \Big\lfloor \frac{(m,\lambda_u)s+t\lambda_u}{m} \Big\rfloor P_u + \sum_{\substack{{i=0}\\{i\neq u}}}^{r+1} \Big\lfloor \frac{t\lambda_i}{m} \Big\rfloor P_i  \\
& = &\Big\lfloor \frac{s+t\eta_u}{m_u} \Big\rfloor P_u +
 \sum_{\substack{{i=0}\\{i\neq u}}}^{r+1} \Big\lfloor \frac{t\lambda_i}{m} \Big\rfloor P_i  
\end{eqnarray*}
where $\eta_u:=\frac{\lambda_u}{(m,\lambda_u)}$ and $m_u:=\frac{m}{(m,\lambda_u)}$. Now,  Theorem \ref{ThMaharaj}  implies 
$$
\begin{aligned}
&\ell((s+1)D_u)-\ell(sD_u)=
&\\
&\sum_{t=0}^{m-1} \left[ \ell \left(\Big\lfloor \frac{s+1+t\eta_u}{m_u} \Big\rfloor P_u + D_t\right) - \ell \left(\Big\lfloor \frac{s+t\eta_u}{m_u} \Big\rfloor P_u  + D_t \right) \right],
\end{aligned}
$$

\noindent where $D_{t} =\sum\limits_{\substack{{i=0}\\{i\neq u}}}^{r+1} \Big\lfloor \frac{t\lambda_i}{m} \Big\rfloor P_i$. Note that $(\eta_u,m_u)=1$, and thus using  the unique $t_u\in \{0,\ldots,m_u-1\}$ given by Lemma \ref{lema-t} we obtain

$$\ell((s+1)D_u)-\ell(sD_u)=\sum_{j=0}^{(m,\lambda_u)-1}\ell_j, \text{ where }$$

\begin{eqnarray*}
\ell_j &=& \ell \left(\Big\lfloor\frac{s+1+(t_u+jm_u)\eta_u}{m_u}\Big\rfloor P_u + D_{t_u+jm_u}\right)\\
&-& \ell \left(\Big\lfloor\frac{s+(t_u+jm_u)\eta_u}{m_u}\Big\rfloor P_u + D_{t_u+jm_u} \right)\\
& = & \ell \left(\left(1+\Big\lfloor\frac{s}{m_u}\Big\rfloor+\Big\lfloor\frac{(t_u+jm_u)\eta_u}{m_u}\Big\rfloor\right) P_u + D_{t_u+jm_u}\right) \\
&-&\ell \left(\left(\Big\lfloor\frac{s}{m_u}\Big\rfloor+\Big\lfloor\frac{(t_u+jm_u)\eta_u}{m_u}\Big\rfloor\right) P_u + D_{t_u+jm_u} \right).
\end{eqnarray*}

\noindent Note that the  last equality implies $\ell_j \in\{0,1\}$. Also, since $K(x)$ has genus zero, Riemann's Theorem implies that $\ell_j =0$ if and only if 
$$\deg\left((1+\Big\lfloor\frac{s}{m_u}\Big\rfloor+\Big\lfloor\frac{(t_u+jm_u)\eta_u}{m_u}\Big\rfloor) P_u + D_{t_u+jm_u}\right)<0.$$ However
$$\deg\left((1+\Big\lfloor\frac{s}{m_u}\Big\rfloor+\Big\lfloor\frac{(t_u+jm_u)\eta_u}{m_u}\Big\rfloor) P_u + D_{t_u+jm_u}\right)=$$
\begin{eqnarray*}
&=& 1+\Big\lfloor\frac{s}{m_u}\Big\rfloor+\Big\lfloor\frac{(t_u+jm_u)\lambda_u}{m}\Big\rfloor + \sum_{\substack{{i=0}\\{i\neq u}}}^{r+1} \Big\lfloor \frac{(t_u+jm_u)\lambda_i}{m} \Big\rfloor\\
&=& 1+\Big\lfloor\frac{s}{m_u}\Big\rfloor +\sum\limits_{i=0}^{r+1} \Big\lfloor \frac{(t_u+jm_u)\lambda_i}{m} \Big\rfloor\\
&=&1+\Big\lfloor\frac{s}{m_u}\Big\rfloor +\sum\limits_{i=0}^{r} \Big\lfloor \frac{(t_u+jm_u)\lambda_i}{m} \Big\rfloor +\Big\lfloor -\sum\limits_{i=0}^{r} \frac{(t_u+jm_u)\lambda_i}{m} \Big\rfloor\\
&=&1+\Big\lfloor\frac{s}{m_u}\Big\rfloor +\sum\limits_{i=0}^{r} \Big\lfloor \frac{(t_u+jm_u)\lambda_i}{m} \Big\rfloor -\Big\lceil \sum\limits_{i=0}^{r} \frac{(t_u+jm_u)\lambda_i}{m} \Big\rceil\\
&=&1+\Big\lfloor\frac{s}{m_u}\Big\rfloor -\Big\lceil\sum\limits_{i=0}^{r} \Big\lbrace \frac{(t_u+jm_u)\lambda_i}{m} \Big\rbrace\Big\rceil.
\end{eqnarray*}
Notice that the last two equalities are due to Remark \ref{lulu}. Therefore $\ell_j=0$  if and only if $1+\Big\lfloor\frac{s}{m_u}\Big\rfloor <\sum\limits_{i=0}^{r} \Big\lbrace \frac{(t_u+jm_u)\lambda_i}{m} \Big\rbrace$, and this completes the proof.
\end{proof}

As an immediate consequence of  Theorem \ref{main} we obtain the next result.

\begin{cor}\label{ponto}
Using the same notation as in Theorem \ref{main}, we have that  if $D_u$ is a totally ramified place in the extension $F/K(x)$, then $s$ is a gap number at $D_u$ if and only if 
$$\sum\limits_{i=0}^r\Big\lbrace\frac{t_u\lambda_i}{m}\Big\rbrace>1+\Big\lfloor\frac{s-1}{m}\Big\rfloor. $$ 
\end{cor}

\begin{proof}
It follows immediately from the fact that $m_u=m$, and then  $\Big\lbrace \frac{(t_u+jm_u)\lambda_i}{m} \Big\rbrace= \Big\lbrace \frac{t_u\lambda_i}{m} \Big\rbrace$.
\end{proof}

Note that Lemma \ref{dimension} together with (\ref{eqmain}) in Theorem \ref{main} tell us that sufficiently small values of $s$ are
gap numbers. We point out  the effect of that on the generic places of $F$; that is, the  places  of  $F$ lying above a place $P_0 \in \mathbb{P}(K(x))$.

\begin{cor}\label{generic} Let  $Q_0$ be  a generic place in $F/K(x)$, and $s\geq 1$ an integer.  If there exists $j \in \{0,1,\ldots,m-1\}$ such that 
$$s <\sum\limits_{i=0}^{r} \Big\lbrace \frac{j \lambda_i}{m} \Big\rbrace,$$
then $s$ is a gap at $Q_0$. In particular,  
\begin{enumerate}
\item[(i)]  If  $s <r/2$ then $s$ is a gap at $Q_0$.
\item[(ii)] If  $\lambda_1=\cdots=\lambda_r$ is coprime to $m$,  and $s<r(m-1)/m$  then  $s$ is a gap at $Q_0$.
\end{enumerate}
\end{cor}

\begin{proof}
The main assertion is a direct application of Theorem \ref{main}, when $m_u=1$ and $\lambda_u=t_u=0$.
The additional assertions follow from Lemma \ref{inteiro} 
 and Lemma \ref{dimension}.
\end{proof}

As previously mentioned, the characterization given by Theorem  \ref{main} 
can be used to extend  independent results by several authors. 
The next section presents an application  in the setting of maximal curves.

\section{Maximal Curves over $\mathbb{F}_{q^2}$}\label{maximal}

A projective, geometrically irreducible, nonsingular algebraic curve  defined over $\F_{q^2}$ of genus $g$ is called $\mathbb{F}_{q^2}$-maximal 
if its  number of $\mathbb{F}_{q^2}$-rational points attains the Hasse-Weil upper bound
$$ 1+q^2  + 2gq.$$
The most well know example  of $\mathbb{F}_{q^2}$-maximal curve is the Hermitian curve $$x^{q+1}+y^{q+1}=1.$$
Curves with many rational points, in particular, maximal curves, have very important applications in Coding Theory  \cite{Sti}, and several other areas, such as, finite geometry \cite{Hirsch}, correlation of shift register sequences \cite{Lidl}  and exponential sums \cite{Moreno}. Maximal curves have been extensively  studied by many authors, not only because of its applications, but also because of  the attractive mathematical challenges represented by the related problems.

To address one of these problems, we recall that a well known  result of Serre (presented by Lachaud in \cite[Proposition 6]{La}), says that any curve which is $\F_{q^2}$-covered by an  $\F_{q^2}$-maximal curve is $\F_{q^2}$-maximal as well. In particular, if $m$ and $n$ are divisors of q+1, then the curve
                        $$x^n+y^m=1,$$
is $\F_{q^2}$-maximal. A natural question is whether or not the converse
of this statement holds true. More generally, one may ask the following:

\text{}\\
{ \bf Question 1.}  For an $\F_{q^2}$-maximal curve of type $y^m=f(x)$, which conditions on the polynomial $f(x) \in \F_{q^2}[x]$ will assure that $m$ divides $q+1$?

\text{}\\

Substantial progress towards an answer of Question 1 has been made by Garcia, Tafazolian and Torres in a sequence of papers \cite{GT2, GT1,T4,TT}. In particular, they proved that $m$ must divide $q+1$ in the following cases.

\begin{enumerate}[\rm 1.]
\item $f(x)=1-x^n$  (\cite[Theorem 4.4]{GT2},\cite[Theorem 5]{TT}).
\item $f(x)=x^2-1$   \cite[Theorem 1]{T4}.
\item $f(x) \in \mathbb{F}_{q^2}[x]$ is an additive and separable polynomial  (\cite[Theorem 2]{GT1},\cite[Theorem 9]{TT}).
\end{enumerate}


The aim of this section is to  answer Question 1 for  a wide class of polynomials $f(x) \in \F_{q^2}[x]$. In this way, we will extend and incorporate previous results listed above.    Our main ingredient will be Theorem \ref{main}. However, part of the approach is based on the same strategy used in \cite{T4,TT}. That is, we make use of the following well known results (see \cite[Lemma 1]{RH} and \cite[Lemma 3]{TT}).

\begin{lemma}\label{basico}
 Let $\mathcal{X}$ be a maximal curve over $\Fq2$ , and let $P$ and $Q$ be two distinct rational points on ${X}$. Then
\begin{enumerate}[\rm(i)]
 \item $(q+1)P\sim (q+1)Q$.
 \item If there exists an integer $m\geq 1$ such that $mP\sim mQ$, then 
 $$d:=\gcd(m,q+1)$$
 is a nongap at $P$.
  \end{enumerate}
\end{lemma}

The following result  uses the same  notation as in Theorem \ref{main}.
\begin{lemma}\label{gap} Suppose $P$ and $Q$ are totally ramified over two distinct roots of $f(x)$ with the same multiplicity, then  $1,2,\ldots, \lfloor m/2\rfloor-1$ are gaps at $P$. Moreover, if  either $m$ or the multiplicity of a third root of $f(x)$ is an odd number, then $1,2,\ldots, \lfloor m/2\rfloor$ are gaps at $P$. 
\end{lemma}

\begin{proof}
Without loss of generality, assume $\gcd(\lambda_1,m)=\gcd(\lambda_2,m)=1$.  If  $s<m/2$, then 
$$\sum\limits_{i=1}^r\Big\lbrace\frac{t_u\lambda_i}{m}\Big\rbrace=\frac{m-s}{m}+\frac{m-s}{m}+\sum\limits_{i=3}^r\Big\lbrace\frac{t_u\lambda_i}{m}\Big\rbrace>1=1+\Big\lfloor\frac{s-1}{m}\Big\rfloor,$$
 and then  Theorem \ref{main} implies that $s$ is a gap at $P$.

If  $m$ is odd, note that $1,2,\ldots, \lfloor m/2\rfloor$ are all smaller then $m/2$, and the result follows from the discussion above.  If $m$ is even, then for   $s= m/2$, we have 
 $$\sum\limits_{i=1}^r\Big\lbrace\frac{t_u\lambda_i}{m}\Big\rbrace=\frac{2(m-s)}{m}+\sum\limits_{i=3}^{r}\Big\lbrace  \frac{(m-s)\lambda_i}{m} \Big\rbrace >1 \text{ if and only if } \sum\limits_{i=3}^{r}\Big\lbrace  \frac{\lambda_i}{2} \Big\rbrace >0.$$ 
 Clearly the latter condition is equivalent to  $\lambda_i$ being  odd for some $i\geq 3$.
\end{proof} 
 
\begin{theorem}\label{main2}
Notation as in Theorem \ref{main}.  If $F/K(x)$ is $\Fq2$-maximal, and $P$ and $Q$ are rational points totally ramified over two distinct roots of $f(x)$  with the same multiplicity, then the following holds:
\begin{enumerate}[\rm(i)]
 \item $m$ divides $2(q+1)$, and 
 \item if  either $m$ or the multiplicity of a third root of $f(x)$ is an odd number, 
  \end{enumerate}
then $m$ divides $q+1$. 
\end{theorem}

\begin{proof}

 Since $P$ and $Q$ are totally ramified, the equation $y^m=f(x)$ gives $mP\sim mQ$. From Lemma \ref{basico}, we have that  $(q+1)P\sim (q+1)Q$  and that   $d=\gcd(m,q+1)$ is a nongap at $P$. Therefore, since  $d$ divides $m$, it follows from Lemma \ref{gap} that $d\in \{m/2,m\}$, and this proves (i). For the second assertion, note that from Lemma \ref{gap} the number $\lfloor m/2\rfloor$ is a gap at $P$.
Therefore, $d\neq m/2$, and then $m=d$ is a divisor of $q+1$.
\end{proof} 
 
\begin{remark} Note that most of the  polynomials $f(x)$  found in \cite{GT1,T4,TT} (listed above)  are cases covered by Theorem \ref{main2} (ii). The only exception is the case that $f(x)=\pm(1-x^2)$  and $m$ is even. The reason is because if $m$ is even,  then one should not expect the result to hold true  for arbitrary polynomials $f(x)=ax^2+bx \in \mathbb{F}_{q^2}[x]$. Indeed, it can be checked that  if  $\mathbb{F}_{81}^*=\langle \alpha \rangle$, then the curve $y^4=x^2-\alpha^2$   is  $\mathbb{F}_{81}$-maximal,  but $4\nmid 10$. Also, in general, the  given condition on a third root of $f(x)$ in Theorem \ref{main2}  (ii) cannot be removed. For instance, once again for $\mathbb{F}_{81}^*=\langle \alpha \rangle$, one can check that  $y^4=x(x-1)(x-\alpha)^2$ is  $\mathbb{F}_{81}$-maximal.

 \end{remark}

\section{Totally ramified Weierstrass points}

In positive characteristic, the  gap sequence  of  a  function field has an important arithmetic property which is recalled in the next Remark. This property turns out to be very useful   to decide   when a totally ramified place is a Weierstrass point.

\begin{remark}\label{condicao} If char$(F)=p>0$,  the gap sequence of $F$ has the following important arithmetic property: if $\lambda_1, \dots, \lambda_g$ is the gap sequence of $F$ and $\mu$ is an integer  that is $p$-adically not greater than $\lambda_i-1$ for some $i$, then $\mu+1$ is also a gap number of $F$(see \cite{V-M}). Such property will be used to study conditions for which a totally ramified place is Weierstrass. 
\end{remark}

\begin{lemma}\label{m+1} Notation as in  Corollary \ref{ponto}. If $P_i$ is a totally ramified place such that $m+1$ is a gap at $P_i$, then $P_i$ is a Weierstrass point of $F$.
\end{lemma}
\begin{proof}
If $P_i$ is not a Weierstrass point of $F$, then its gap sequence is the actual gap sequence of the field $F$. This imples that  $m+1$ is in the gap sequence of $F$. The definition of $m$ gives $(m,p)=1$, and so one can write $m=cp+d$  where $c$ and $d$ are non-negative integers, and $1 \leq d \leq p-1$. Note that the $p$-adic coefficients of $m-1=cp+(d-1)$ are not greater than the corresponding coefficients of $m$, and by the arithmetic condition (cf. Remark \ref{condicao})\, $m$ is in the gap sequence of $F$. In particular, $m$ is a gap at $P_i$. However, this is a clear contradiction  to the (easy to check) fact that  $(\frac{1}{x-\alpha_i})_\infty = mP_i$ for $i=1, \dots, r$ and $(x)_\infty =mP_\infty$ , which finishes the proof.
\end{proof}

The following illustrates how the characterization given in Corollary \ref{ponto} can be used to derive various conditions under which all totally ramified places are Weierstrass. Note that item $(i)$ below extends a Valentini-Madan's result  (Theorem $3$ in \cite{V-M}).

\begin{cor}\label{ugly}
Using the notation of Corollary \ref{ponto}, let $k\geq 1$ be  the number of   totally ramified places of $F$, and $w.l.o.g.$ assume that $\Gamma=\{\lambda_1, \ldots,\lambda_k\}$ is the set of multiplicities of the roots of $f(x)$ associated to  such places.  Then the totally ramified places of $F$ are Weierstrass  if  any the following holds:
\begin{enumerate}
\item[(i)]  $r \geq m+3.$
\item[(ii)]  $3\leq k\leq r\leq 4$ and $\#\Gamma=1$ or $3 \leq k$ and $\#\Gamma=1$ and $ 4 \leq m$.
\item[(iii)]  $r\geq5$ and the map $\varphi: \{1,\ldots,k\}\longrightarrow \Gamma $  given by $\varphi(i)=\lambda_i$ satisfies $\#\varphi^{-1}(\{\lambda_j\})\geq2$ for $j=1,\ldots,k$. 
\item[(iv)] $r=\phi(m)>4$ and $\Gamma=(\Z/(m))^{\times}$, where
$\phi$ is the Euler's totient function\footnote{One can actually pick any subgroup  $\Gamma \subseteq (\Z/(m))^{\times}$, of order greater than $4$, such that $m-1 \in \Gamma $.}
\end{enumerate}
\end{cor}
\begin{proof}
Let $P_i$ be a totally ramified place of $F$ and $t \in \{1,\ldots,m-1\}$ be such that $t\lambda_i \equiv m-1 \mod m$. For $j=1,\ldots,k$, let $r_j$ be the remainder of the division of $t\lambda_j$ by $m$. Clearly $r_i=m-1$, and  we claim that $r_j\geq 1$ for all $j$. In fact, first notice that
$t\lambda_i \equiv m-1 \mod m$ implies $(t,m)=1$. Now if there is a $j$ such that $r_j=0$, then $m$ divides $t\lambda_j$; that is, $m$ divides $\lambda_j$, which contradicts $\lambda_j<m$. Moreover, since $\bigg\lbrace \frac{t\lambda_j}{m}\bigg\rbrace=\frac{r_j}{m}$, based on Corollary \ref{ponto} and Lemma \ref{m+1} we only need to prove that $m+1$ is a gap at $P_i$, that is:   $$\sum\limits_{j=1}^{r}\Big\lbrace \frac{t\lambda_j}{m}\Big\rbrace=\frac{m-1}{m}+\sum\limits_{j\neq i}\frac{r_j}{m}>2.$$

\begin{enumerate}

\item[(i)] Since  $r_j \geq 1$ we get 
 $$\frac{m-1}{m}+\sum\limits_{j\neq i}\frac{r_j}{m}\geq \frac{m+r-2}{m}>2,$$ as desired. 
 

\item[(ii)] Since $\#\Gamma =1$,  there exist $i_1,i_2 \in \{1,
\ldots,k\}$, with $i, i_1$ and $i_2$  all distinct, such that $r_i=r_{i_1}=r_{i_2}=m-1$. Hence,
$$\frac{m-1}{m}+\sum\limits_{j\neq i}\frac{r_j}{m}=3\cdot\frac{m-1}{m}+\sum\limits_{j\notin \{i,i_1,i_2\}}\frac{r_j}{m}>2,$$
and again  the result follows.

\item[(iii)] The condition on the map $\varphi $ implies that one can find $i_1 \in \{1,
\ldots,k\}$, with $i_1\neq i$, such that $r_i=r_{i_1}=m-1$. Thus 
$$\frac{m-1}{m}+\sum\limits_{j\neq i}\frac{r_j}{m}=2\cdot\frac{m-1}{m}+\sum\limits_{j\notin \{i,i_1\}}\frac{r_j}{m}\geq \frac{2m-2+k-2}{m}>2.$$

\item[(iv)] Since $(t,m)=1$,  the map $\lambda_j \mapsto  t\lambda_j $ is just a permutation on $\Gamma$. Therefore,
$$\sum\limits_{j=1}^{k}\Big\lbrace \frac{t\lambda_j}{m}\Big\rbrace=\sum\limits_{j=1}^{k}\frac{r_j}{m}=\sum\limits_{j=1}^{k}\frac{\lambda_j}{m}=\frac{m\phi(m)}{2m}>2.$$

\end{enumerate}
\end{proof}

Note that from Corollary \ref{generic} and Corollary \ref{ugly}.(iii), we recover a characterization of Weierstrass points on hyperelliptic curves. In fact, consider the curve $y^2=(x-\alpha_1)\cdots (x-\alpha_r)$ of  genus  $g=\bigg\lfloor \frac{r-1}{2} \bigg\rfloor$, $r \geq 5$. From Theorem \ref{main}: $t_u\in \{0,1\}$ and $t_u \equiv -s \mod 2$. Therefore from Corollary \ref{ponto} a number  $s$ is a gap at $P$ if, and only if,  $s$ is odd and 
 
$$\frac{r}{2}=\sum\limits_{i=1}^r\Big\lbrace\frac{1}{2}\Big\rbrace>1+\Big\lfloor\frac{s-1}{2}\Big\rfloor.$$
That is, the gap numbers at $P$ will be the $\Big\lfloor\frac{r-1}{2}\Big\rfloor=g$ odd numbers $s<r-1$.

From Corollary \ref{generic}, the numbers $1,2,\ldots,g$ (which are $<r/2$) are gap number at any unramified point $P$. Since this sequence has length $g$, these are all the gap numbers at $P$. In particular  Hyperelliptic curves are  { \bf classical}. With this discussion, we retrieve (see Satz 8 in \cite{Sch}) the following result:

\begin{theorem}[F. K. Schmidt]
 Let $C$ be a hyperelliptic curve defined over a perfect field of characteristic $p$. If $P$ is a point on $C$, then we have the
following:
\begin{itemize}
\item  $P$ is a Weierstrass point if and only if $P$ is a ramification point of the hyperelliptic cover (or, equivalently, $P$ is a fixed point of the hyperelliptic involution).

\item If $P$ is a non-Weierstrass point on the curve $C$, then $P$ has gap sequence $1,\ldots, g$ and so $H(P) = \{0, g + 1, g + 2,\ldots\}$.

\end{itemize}

\end{theorem}

\section{A special Kummer extension}\label{sep}

In this section, we consider only Kummer extensions $F/K(x)$ where $\lambda_i=1$ for all $i$.
We begin by noticing that if $\deg f(x)=m$, then by Corollary \ref{generic} all the integers $1, \dots, m-2$ are gaps at a generic place of $F/K(x)$. This generalizes Leopoldt's result \cite{Leo}. For the totally ramified places, we have the following:

\begin{theorem}\label{pesos}
If the genus $g \geq 1$ and $P\neq P_\infty$ is a totally ramified place, then the set of gap numbers at $P$ is given by
$$G(P)=\Big\{1+i+mj \,|\, 0 \leq i \leq m-2-\bigg\lfloor m/r\bigg\rfloor, \, 0\leq j \leq r-2- \bigg\lfloor\frac{r(i+1)}{m}\bigg\rfloor\Big\}.$$
\end{theorem}

\begin{proof}
Since $\lambda_i=1$ for all $i=1,\ldots, r$, it follows from Theorem \ref{main} that $t:=t_1=\cdots=t_r=m-s \mod m$, where $t\in\{0,\ldots,m-1\}$. Writing $s-1=m\bigg\lfloor\frac{s-1}{m}\bigg\rfloor+k, 0 \leq k \leq m-1$, we have that
$m-s=m(-\bigg\lfloor\frac{s-1}{m}\bigg\rfloor)+m-k-1$, and since $0 \leq m-k-1 \leq m-1$, we conclude $\bigg\lbrace\frac{t}{m}\bigg\rbrace=\frac{m-1-k}{m}$.

Thus, from Corollary \ref{ponto}, $s=m\Big\lfloor\frac{s-1}{m}\Big\rfloor+k+1$ is a gap at $P$ if and only if
$$\Big\lfloor\frac{s-1}{m}\Big\rfloor+ 1<\frac{r(m-1-k)}{m},$$
that is, $\bigg\lfloor\frac{s-1}{m}\bigg\rfloor<r-1-\frac{r(k+1)}{m}$. Therefore, the gap numbers at $P$ will be of the form $s=mj+i+1$, where $i\in \{0,\ldots,m-1\}$ and $0\leq j < r-1-r(i+1)/m$. However, from  the latter inequality we have $r-1-r(i+1)/m>0$, that is $i<m-1-m/r$. Hence, since $i$ and $j$ are integers, we have $i\in \{0,\ldots,m-2-\bigg\lfloor m/r \bigg\rfloor\}$ and $j\in \{0,\ldots,r-2- \bigg\lfloor\frac{r(i+1)}{m}\bigg\rfloor\}$, which gives the result.
\end{proof}

The next corollary extends a Hasse's result \cite{Ha}: ``The totally ramified points on the Fermat curve $y^n=ax^n+b$ have Weierstrass weight  $(n-3)(n-2)(n-1)(n+4)/24.$"
 
\begin{cor}\label{hasse} If the curve $y^m=f(x)$ is  smooth and  classical, $\deg f(x)=r$, then any totally ramified place not equal to $P_\infty$, has Weierstrass weight
$$(r-3)(r-2)(r-1)(r+4)/24.$$
\end{cor}
\begin{proof}
 
Since $f(x)$ is separable, the curve $y^m=f(x)$  is smooth if and only if $r\in\{m-1,m,m+1\}$. From Theorem  \ref{pesos}, if $r\in\{m-1,m\}$ (the case $r=m+1$ will follow analogously) then we have \footnote{Note that $ \#G(P)=\sum_{i=0}^{m-3} \left( m-2- i\right)=(m-1)(m-2)/2= g,
$ as expected.}
$$G(P)=\{1+i+mj \, |\, 0 \leq i \leq m-3, \quad 0\leq j \leq m-3-i\}.$$
The Weierstrass weight at $P$ is
\begin{eqnarray*}
W(P) &=&\sum\limits_{\lambda\in G(P)}\lambda-g(g+1)/2\\
& = &\sum\limits_{i=0}^{m-3}\sum\limits_{j=0}^{m-3-i}(1+i+mj) -g(g+1)/2\\
 &=&(m-3)(m-2)(m-1)(m+4)/24,
\end{eqnarray*}
where  $g=(m-1)(m-2)/2$.
\end{proof}

We point out the $G(P)$ in the proof above can also be obtained by looking at the $(\mathcal{D},P)$-orders   for suitable linear systems  of type $\mathcal{D}=|dP|$.

\subsection{The number of certain Weierstrass points}
In this section, $f(x) \in K[x]$  will be a separable polynomial, and the Kummer extension $y^m=f(x)$ will be classical.

Characterizing Weierstrass points as well as estimating their weights is a fundamental problem in the theory of algebraic curves. In the case of  classical curves,  a fundamental result by Lewittes \cite{Lew} says  that if the curve has an automorphism that fixes at least five points, then all of the fixed points must be Weierstrass points. For 
example,  hyperelliptic curves have an involution fixing all branch points, thus it follows that all such points are  Weierstrass points. As we know, these points  are the whole of the Weierstrass points on a hyperelliptic curve. 

Hyperellipic curves can be generalized by  the curves of type $y^m=f(x)$, where $f(x)$ is a polynomial of distinct roots. Clearly, the map $(x,y)\mapsto (x,\zeta_ny)$, where $\zeta_n$ is a $n^{th}$ primitive roots of unity, is an automorphism of these curves, and it fixes all branch  points. Therefore, in the classical case, all  branch points of these curves are Weierstrass points. In this case, it is also  well known that  total number of Weierstrass points, counted with multiplicity, is $g^3-g$. where $g$ is the genus of the curve.

 A natural question, raised by Towse in 1996 \cite{To}, is how to relate the   total weight, namely $BW$, of all of the branch points of $y^m=f(x)$, with to the total weight $g^3-g$ of 
 Weierstrass points on these curves. For this, Towse studied the asymptotic behavior of the ratio $BW/(g^3-g)$, as $r:=\deg f(x)$ goes to infinity. Towse's main result  in  \cite{To} states that
 
 $$\lim \inf_{r\rightarrow\infty}\frac{BW}{g^3-g}\geq \frac{1}{3}\frac{m+1}{(m-1)^2},$$
where $r=\deg f(x)$.

%
Now, using Theorem \ref{pesos}, we prove that  Towse's result can be  extended to the following:

\begin{cor} Under the assumptions of Corollary \ref{hasse}, we have that:
$$\lim_{r\rightarrow\infty}\frac{BW}{g^3-g}=\frac{1}{3}\frac{m+1}{(m-1)^2}.$$
\end{cor}

\begin{proof}
Considering $r>m$, it follows from Theorem \ref{pesos} that the sum $S(P)$ of the gaps at any totally ramified place $P\neq P_{\infty}$ is given by:
$$S(P)=\sum\limits_{\lambda\in G(P)}\lambda=\sum\limits_{i=0}^{m-2}\quad \sum\limits_{j=0}^{r-2-\lfloor\frac{r(i+1)}{m}\rfloor}(1+i+mj)=$$
\begin{eqnarray*}
&=&\sum\limits_{i=0}^{m-2}\left[(r-1-\Big\lfloor\frac{r(i+1)}{m}\Big\rfloor)(1+i+\frac{m}{2}(r-2-\Big\lfloor\frac{r(i+1)}{m}\Big\rfloor)\right]\\ 
&=&\sum\limits_{i=0}^{m-2} \left[(r-1)(1+i)\right]+\sum\limits_{i=0}^{m-2} \left[ (r-1) (\frac{m}{2}(r-2-\Big\lfloor\frac{r(i+1)}{m}\Big\rfloor))\right] \\
&-&\sum\limits_{i=0}^{m-2} \left[\Big\lfloor\frac{r(i+1)}{m}\Big\rfloor(1+i+\frac{m}{2}(r-2))\right]+\frac{m}{2} \sum\limits_{i=0}^{m-2} \Big\lfloor\frac{r(i+1)}{m}\Big\rfloor^2.\\
\end{eqnarray*}
Using that $\sum\limits_{i=0}^{m-2} \bigg\lfloor(\frac{r(i+1)}{m})\bigg\rfloor = \frac{(r-1)(m-1)}{2}$ and $\bigg\lfloor\frac{r(i+1)}{m}\bigg\rfloor= \frac{r(i+1)}{m}-\left\{\frac{r(i+1)}{m}\right\}$, we obtain:
\begin{eqnarray*}
S(P)= O(r)+\frac{(2m-1)(m-1)}{12}r^2,\\
\end{eqnarray*}
where $O(r^t)$ denotes a polynomial in the variable $r$ of degree less or equal than $t$.
Therefore, since $\frac{g(g+1)}{2}=\frac{(m-1)^2}{8}r^2+O(r)$, the weight $W(P)$ at $P$ is given by
$$W(P)=S(P)-\frac{g(g+1)}{2}=\frac{m^2-1}{24}r^2+O(r).$$
Since $g^3-g=\frac{(m-1)^3}{8}r^3+O(r^2)$, we have (note that $W(P_{\infty})$ is irrelevant for the limit):
$$\lim_{r\rightarrow\infty}\frac{BW}{g^3-g}=\lim_{r\rightarrow\infty}\frac{rW(P)}{g^3-g}=\frac{(m^2 - 1)/24}{(m-1)^3/8}=\frac{1}{3}\frac{m+1}{(m-1)^2},$$
as desired.

\end{proof}

\end{document}